\documentclass[12pt]{amsart}

\usepackage{graphicx}
\usepackage{amscd}
\usepackage{amsmath}
\usepackage{amsfonts}
\usepackage{amssymb}
\usepackage{a4wide}
\usepackage{setspace}
\usepackage{enumerate}         % better lists
\usepackage{color}
\usepackage{url}
\usepackage{amsthm}
\usepackage{hyperref}
\usepackage{bm}
\usepackage{xy}
\usepackage{stmaryrd}

\theoremstyle{plain}
\newtheorem{theorem}{Theorem}[section]

\newtheorem{corollary}[theorem]{Corollary}

\newtheorem{lemma}[theorem]{Lemma}

\newtheorem{proposition}[theorem]{Proposition}

\newtheorem{definition}[theorem]{Definition}

\newtheorem{assumption}[theorem]{Assumption}
\theoremstyle{remark}
\newtheorem{remark}[theorem]{Remark}

\numberwithin{equation}{section}

\newcommand{\ind}{1\!\kern-1pt \mathrm{I}}
\newcommand{\rsto}{]\!\kern-1.8pt ]}
\newcommand{\lsto}{[\!\kern-1.7pt [}

\numberwithin{equation}{section}

\newcommand{\RR}{\mathbb{R}}

\newcommand{\NN}{\mathbb{N}}

\newcommand{\cC}{\mathcal{C}}
\newcommand{\cD}{\mathcal{D}}

\newcommand{\cF}{\mathcal{F}}
\newcommand{\cM}{\mathcal{M}}

\newcommand{\cY}{\mathcal{Y}}

\newcommand{\massP}{\mathbf{P}}
\newcommand{\massQ}{\mathbf{Q}}
\newcommand{\massE}{\mathbf{E}}
\newcommand{\massL}{L}

\newcommand{\linftys}{(L^\infty)^*}
\newcommand{\linftysp}{(L^\infty)^*_+}
\newcommand{\infseq}{\}_{n=1}^\infty}

\newcommand{\conv}{\textnormal{conv}}

\begin{document}
\title[Dual problem of utility maximization] {On the dual problem of utility maximization\\ in incomplete markets}
\author{Lingqi Gu}
\author{Yiqing Lin}
\author{Junjian Yang}
\address{\noindent University of Vienna\newline \indent Faculty of Mathematics\newline\indent Oskar-Morgenstern Platz 1\newline \indent A-1090 Wien, Austria\newline}
\email{yiqing.lin@univie.ac.at}

\begin{abstract}
In this paper, we study the dual problem of the expected utility maximization in incomplete markets with bounded random endowment. We start with the problem formulated in \cite{CSW01} and prove the following statement: in the Brownian framework, the countably additive part $\widehat{Q}^r$ of the dual optimizer $\widehat{Q}\in (\massL^\infty)^*$ obtained in \cite{CSW01} can be represented by the terminal value of a supermartingale deflator $Y$ defined in \cite{KS99}, which is a local martingale. 
\end{abstract}

% \subjclass[2010]{Primary}
\date{\today}
\thanks{The authors gratefully acknowledge support from the European Research Council (ERC) under grant FA506041. }
\maketitle 

% \newpage

\section{Introduction}
\noindent Optimal investment is a classical problem in mathematical finance, which concerns an economic agent who invests in a financial market so as to maximize the expected utility of his terminal wealth. In this paper, we consider the utility maximization problem in general semimartingale markets and focus on the primal-dual approach, precisely, we are interested in properties of the solution to the dual problem.
%Roughly speaking, the authors consider an agent endowed with a deterministic initial wealth $x$ and aiming to maximize the expected utility of his terminal wealth 
 %$\massE[U(X_T)]$ (here, $T$ denotes the terminal time) 
 %by trading with admissible strategies. 

\vspace{2mm}

 A complete review of the literature on optimal investment is too extensive, so we only concentrate on those of immediate interest. In the semimartingale framework, 
 %the utility maximization problem with general semimartingale settings in incomplete markets. In this framework, 
 Kramkov and Schachermayer study the problem with a utility function $U$ supported on $\mathbb{R}_+$ in \cite{KS99}. They assume the agent is endowed with a deterministic initial wealth $x$ and
 %$\massE[U(X_T)]$ (here, $T$ denotes the terminal time)  
 can trade using admissible strategies
  - those which keep the corresponding wealth process bounded from below. 
 The core approach in \cite{KS99} is to construct an optimal trading strategy by solving a dual problem defined on $L^0$. The result in \cite{KS99} is subsequently generalized by Cvitani\'c et al.~\cite{CSW01} to allow for receiving a bounded random endowment $e_T$ other than the deterministic initial wealth, which, for example, can be an exogenous non-traded European contingent claim. In this case, the primal problem under consideration is formulated as
$$
u(x):=\sup_{g\in \mathcal{\mathcal{C}}}\massE[U(x+g+e_T)], \quad x\in \mathbb{R},$$
where $\mathcal{C}$ is the convex cone of all random variables dominated by admissible stochastic integrals. The authors of \cite{CSW01} employ the duality between $L^\infty$ and $\linftys$
%the space of finitely additive signed measures 
and solve a minimization dual problem over the subset $\mathcal{D}$ of $\linftys$, which can be regarded as the weak-star closure of the set $\cM$ of equivalent local martingale measures (ELMMs). Precisely, the dual problem is formulated as 
\begin{equation}\label{posepro}
v(y):=\inf_{Q\in \mathcal{D}}\left\{\massE\left[V\left(y\frac{dQ^r}{d\massP}\right)\right]+y\langle Q, e_T\rangle\right\},\quad y>0, 
 \end{equation}
 where $V$ is the conjugate of $U$, $\massP$ is the physical probability measure, and $Q^r$ is the regular part of $Q\in\linftys$.
 %which can be regarded as the weak-star closure of the set $\cM$ of equivalent local martingale measures (ELMMs). 
 It states in \cite{CSW01} that a dual optimizer $\widehat{Q}$ can be found  in $\mathcal{D}$, which is unique up to the singular part, and moreover the primal optimizer can be formulated in terms of $\widehat{Q}^r$.\\
 
%\\ . 
%construct 
 %consider the following primal problem, 
 In \cite{HK04}, the authors relax the boundedness assumption on the random endowment and instead only require an integrability condition. See also \cite{Kha13} and \cite{GKM14}. 
 % is assumed and the authors stick to the $L^0$ duality similar to \cite{KS99}. 
 Another extension of \cite{CSW01} is provided by Karatzas and \v{Z}itkovi{\'c} in \cite{KZ03} by allowing for intertemporal consumption. \\

  Schachermayer treats the case of utility functions supporting both positive and negative wealth for a locally bounded risky asset without random endowment in \cite{Sch01}.
  %ty maximization problem is treated by Schachermayer with a locally bounded underlying asset price $S$ in \cite{Sch01}. 
  Within this locally bounded semimartingale framework, Owen obtains a generalized result in \cite{Owe02} with a bounded random endowment and furthermore, Owen and \v{Z}itkovi\'c \cite{OZ09} consider unbounded claims. It is worth mentioning that with the settings of \cite{Sch01}, the dual optimizer is always a martingale measure but may not be equivalent, stated for example in \cite{BF02, Owe02, Sch01}. On the other hand, Biagini and Frittelli investigate a problem similar to \cite{Sch01} and allow for a general semimartingale model for the risky asset. Their utility maximization problem is established on a new domain in \cite{BF05} and is afterwards embedded in Orlicz spaces in \cite{BF08}. Eventually, Biagini et al.~\cite{BFG11} study the problem with unbounded random endowment by generalizing  the duality method in \cite{BF08}.\\
  %  on a new domain and 
  
 %  moreand  also . are obtained by Biagini, Frittelli, Owen, \v{Z}itkovi\'c, etc. (see, \cite{BF05}, \cite{OZ09}), where  and \cite{Owe02} and \cite{OZ09} devote to the case with random endowment.

%In the present paper, we concern the property of the supermatingale deflator associated with the solution of the dual problem (\ref{posepro}). 
Besides the duality between the primal and dual problem, many authors are concerned with  the properties of the dual optimizer in the case when $U$ is defined on $\mathbb{R}_+$; particularly, the following representation conjecture is studied: {\it the dual optimizer can be attained by an equivalent local martingale deflator (for short ELMDs, see e.g.~\cite{KK07, Kar12}).}  When $e_T=0$,
% studied in \cite{KS99}, where
 %the dual problem is set up by taking infimum over the set of all terminal values of supermartingale deflators.
 %and % which is indeed the closure of $\mathcal{M}$ in $\massL^0$ w.r.t. the topology of convergence in measure. 
Kramkov and Schachermayer observe that this conjecture fails for general semimartingale models, since there is an example showing that the associated process could only be a strict supermartingale (cf. Example 5.1' in \cite{KS99}).
% fthe dual optimizer may not be found in the set containing only the terminal value of equivalent local martingale deflators  . % since a counterexample shows that the associated process could only be a strict supermartingale.
%In particular, the authors show a example the dual optimizer has lost some mass such that it can only be a process could only be a strict supermartingale.
However, once further conditions are imposed on $S$, one could have some positive results. For example,
% (which is called equivalent local martingale deflator, for short ELMD, see \cite{KK07}). 
%In particular, 
Karatzas and \v{Z}itkovi{\'c} observe that this conjecture is  true for It\^o-process models in \cite{KZ03}, and in a recent paper of Horst et al.~\cite{HHIRZ14}, this dual optimal process is constructed via solutions of backward stochastic differential equations. Another important result is presented in \cite{LZ07}:  the representation conjecture is true for all continuous semimartingale models.\\
%it suffice to assume the continuity of $S$.
% is suffice only assume that $S$ is Larsen and  \v{Z}itkovi{\'c} investigate general continuous semimartingale models and a more general case and obtain this representation with .\\

In the present paper, we study the regular part of the dual optimizer to (\ref{posepro}) and establish the following main result:
 \vspace{2mm}

 {\it If the underlying filtration is Brownian, then the regular part $\widehat{Q}^r$ of the dual optimizer $\widehat{Q}$ to (\ref{posepro}) can be attained by an equivalent local martingale deflator (ELMD).}
 
 \vspace{2mm}

Overall speaking, the present work generalizes the result in \cite{LZ07} to the case that $e_T$ is a bounded random variable,
%the present work can be regarded as an extension of the result in \cite{LZ07} to the case that $e_T$ is a bounded process, 
%with bounded random endowment, 
which increases the complexity of the dynamics.
However, on the technical side,  our proof differs from the existing one. Indeed, the method of Larsen and  \v{Z}itkovi{\'c} is based on the representation theorem in \cite{DS95} for continuous arbitrage-free price processes and the multiplicative decomposition of supermartingale deflators. The authors examine the decomposition of the dual optimal process, and find that the  decreasing part vanishes and only the local martingale part left due to the maximality. By contrast, the dual domain in our case is extended and the definition of the dual problem (\ref{posepro}) involves $e_T$, so we even do not know a priori whether the dual optimizer is associated with a supermartingale deflator, and thus the argument via the decomposition of such processes is not easy to apply. 

 \vspace{2mm}

\vspace{2mm}
Indeed, our approach involves the study of both the primal and dual problems. 
%We explain briefly our idea to prove the main result: 
We construct an optional strong supermartingale $\widehat{Y}$ associated with $\widehat{Q}^r$ using \cite{CS15strong} and a primal optimal process $\widehat{W}$, in particular, the random endowment $e_T$ is endowed with a fictional value process ${e}$. Then, we choose a proper sequence of stopping times $\{\tau_k\}_{k\in \mathbb{N}}$ such that $\widehat{W}_{\tau_k}$ is bounded away from 0. Comparing the expectation of $\widehat{Y}_{\tau_k}$ under $\massP$ and the expectation of $\widehat{W}_{\tau_k}$ under 
%rtingale property of $\widehat{W}$ 
the finitely additive measure $\widehat{Q}$, 
we observe that during the evolution of $\widehat{Y}$, there is no mass that disappears until ${\tau_k}$ and so that $\widehat{Y}$ is a local martingale.
%R\begin{itemize}
%\%item the primal optimal process $\widehat{W}$ is bounded away from 0 at $\tau_k$;
%\item the random variable $\widehat{Y}_{\tau_k}$ is the Radon-Nikodym desity of $(\widehat{Q}|{\mathcal{F}}_{\tau_k})^r$;
%\item the product $\widehat{Y}\widehat{W}$ is a uniformly integrable martingale;
%\item the primal optimal process is a ``martingale'' under the finitely additive measure $\widehat{Q}$. 
%\end{itemize}
%Thus, we observe that at each $\tau_k$, there is no mass disappears and so that $\widehat{Y}$ is a local martingale.

\vspace{2mm}

In this framework, even assuming that $S$ is a geometric Brownian motion, the dual optimal process may fail to be a true martingale (cf. Proposition 2.4 in Larsen et al.~\cite{LSZ15}).
%A constructive example can be found in \cite{} 

\vspace{2mm}

The rest of this paper is organized as follows: in the next section, we recall basic settings and main results in \cite{CSW01}; Section 3 is devoted to the introduction of our main theorem  while Section 4 contains the detailed proof. Finally, Section 5 provides with some technical results and an alternative proof of Proposition 3.2 in \cite{LZ07}.
% with alternative method.
%Tin in the framework of \cite{KS99}.

%[KS99] [KS03]... \\
%\cite{CSW01} [HK04] random endowment\\
%KZXX... consumption\\
%[XX] [Owa] whole real line...\\
%[XX] [XX] BSDE
%\cite{KS99} example not a local martingale...solved as \cite{CSW01} put on the singular part...quote [Pisa]
%Observation of local martingale
%[KZXX] ito process
%[LZXX] continuous process
%[XX] [XX] BSDE specific utility function
%What we do 
%Similar with [LZXX]
%Organization of the paper.

\section{Formulation of the problem}
 \noindent In this section, we shall recall the formulation of the utility maximization problem in incomplete markets with random endowment and briefly introduce the results obtained in \cite{CSW01}. 
 
 \vspace{2mm}

Consider the model of a financial market consisting of $d+1$ assets: one bond and $d$ stocks. Without loss of generality, we assume that the bond price is constant. The stock-price process $S=(S^i)_{1\leq i\leq d}$ is a strictly positive semimartingale on a filtered probability space $(\Omega, \cF, (\cF_t)_{0\leq t\leq T}, \massP)$ 
  satisfying the usual hypotheses of right continuity and saturatedness, where $\cF_0$ is assumed to be trivial. Here, $T$ is a finite time horizon. Unless otherwise specified, we employ the notations $\massL^1$, $\massL^\infty$ and $(\massL^\infty)^*$ to denote the corresponding spaces based on $(\Omega, \mathcal{F}_T, \massP)$. 
  
  \vspace{2mm}
  
 Assume that the agent is endowed with initial wealth $x\in\mathbb{R}$ and his investment strategy is denoted by $H=(H^i)_{1\leq i\leq d}$, which is a predictable $S$-integrable process specifying the number of shares of each stock held in his portfolio. We also assume that the agent receives an exogenous endowment $e_T$ at time $T$, which is $\cF_T$-measurable and satisfies $\rho:=\|e_T\|_{\infty}<\infty$. Then, the total value of his portfolio at time $T$ can be written into
 $$
 W_T=x+(H\cdot S)_T+e_T,
 $$
 where $(H\cdot S)_t=\int^t_0H_udS_u$ denotes the stochastic integral with respect to $S$.
 
 \vspace{2mm}
 
 We call $H$ an admissible strategy if the process $(H\cdot S)$ is uniformly bounded from below by a constant, and we denote by $\mathcal{C}_0$ the convex cone of $\mathcal{F}_T$-measurable random variables dominated by admissible stochastic integrals, i.e.,
 $$
 \cC_0:=\{g\in \cF_T: g\leq (H\cdot S)_T,\ {\rm for\ some\ admissible\ strategy}\ H\}.
 $$ 
 Moreover, we define $\cC:=\cC_0\cap \massL^\infty$.
 
 \vspace{2mm}

Suppose the agent's preferences over terminal wealth are modeled by a utility function $U:(0,\infty)\to\RR$, 
  which is strictly increasing, strictly concave, continuously differentiable and satisfies the Inada conditions:
  $$ U'(0) := \lim_{x\to 0}U'(x) = \infty \quad\textnormal{ and } \quad U'(\infty) := \lim_{x\to \infty}U'(x) = 0. $$
 %Without loss of generality, we may assume $U(\infty)>0$ to simplify the analysis. 
 Without loss of generality, we may assume $U(\infty)>0$ and define $U(x)=-\infty$, if $x\leq 0$. Then, the primal problem can be formulated in the following way:
\begin{equation}u(x)=\sup_{g\in \cC_0}\massE[U(x+g+e_T)],\quad\ x\in \mathbb{R}.\label{pp}\end{equation}
% where $\cC'_0$ consists of those elements in $\cC_0$ such that the expectation above is well defined. 
 
 \vspace{2mm}
 
We adopt the following assumption as in \cite{CSW01, KS99}, which 
%definition of an equivalent local martingale measure (ELMM) from \cite{KS99} and t
 %The following assumption 
 ensures a NFLVR setting (see \cite{DS94, DS98}).
 \begin{assumption}\label{NA}
 There exists at least one probability measure $\massQ\sim\massP$, such that for any $H$ admissible, $(H\cdot S)$ is a local martingale under $\massQ$. Namely, the set $\mathcal{M}$  of all ELMMs is not empty. 
 \end{assumption}  
 
To establish the dual problem, we first define the dual domain, which is a non-empty subset of $(\massL^\infty)^*_+$, convex and compact with respect to the weak-star topology $\sigma((\massL^\infty)^*,\massL^\infty)$:
\begin{equation}\label{dualdom}
\cD:=\{Q\in (\massL^\infty)^*_+: \|Q\|_{(\massL^\infty)^*}=1\ {\rm and}\ \langle Q, g\rangle \leq 0,\ {\rm for\ all}\ g\in \cC\}.
 \end{equation}
 
In fact, the space $(\massL^\infty)^*$ can be identified with the space of bounded additive measures. Each element in $(\massL^\infty)^*_+$ admits a unique Yosida-Hewitt decomposition $Q=Q^r+Q^s$, where the regular part $Q^r\in \massL^1$ is countably additive and the singular part $Q^s$ is purely finitely additive (see \cite{YH52}). Then, the dual problem can be formulated as
\begin{equation}\label{dp}
v(y):=\inf_{Q\in \cD}\left\{\massE^\massP\left[V\left(y\frac{dQ^r}{d\massP}\right)\right]+y\langle Q, e_T\rangle\right\},\quad y>0,
\end{equation}
where $V$ is the conjugate of $U$.

\vspace{2mm}

Throughout this paper, we assume moreover the following assumptions, which ensure the existence of solutions of the primal and dual problems. Detailed discussion on these assumptions  can be found in \cite{CSW01, KS99, KS03}.
%, \cite{KS03} and \cite{CSW01}:
 \begin{assumption} \label{U(x)assumption}
 The utility function $U$ satisfies the reasonable asymptotic elasticity, i.e.,
 $$ AE(U):= \limsup_{x\to\infty}\frac{xU'(x)}{{U}(x)}< 1. $$
\end{assumption}
\begin{assumption}\label{uass}
$|u(x)|<\infty$ holds for some $x>\rho$.
\end{assumption}
Now, we summarize the result obtained in \cite{CSW01} as the following theorem:
\begin{theorem}[Theorem 3.1 and Lemma 4.4 in \cite{CSW01}]\label{thmcsw}
  Under Assumptions \ref{NA}, \ref{U(x)assumption}, \ref{uass},
  we have
  \begin{enumerate}
   \item The primal value function $u$ is finitely valued and continuously differentiable on $(x_0,\infty)$, and $u(x)=-\infty$, for all $x<x_0$, where $x_0:=\sup_{Q\in\cD}\langle Q,-e_T\rangle$.
   \item The dual value function $v$ is finitely valued and continuously differentiable on $(0,\infty)$.
   \item The functions $u$ and $v$ are conjugate in the sense that 
        \begin{equation*}   \label{v=u-xy}
            v(y) = \sup_{x>x_0}\{u(x)-xy\}, \quad y>0,
        \end{equation*}
        \begin{equation*}   \label{u=v+xy}
            u(x) = \inf_{y>0}\{v(y)+xy\}, \quad x>x_0.
        \end{equation*}
   \item For all $y>0$, there exists a solution $\widehat{Q}_y\in\cD$ to the dual problem, which is unique up to the singular part. 
         For all $x>x_0$, $\widehat{g}:= I\Big(\hat{y}\frac{d\widehat{Q}^r_{\hat{y}}}{d\massP}\Big)-x-e_T$ is the solution to the primal problem, 
         where $I=-V'$ and $\hat{y}=u'(x)$, which attains the infimum of $\{v(y)+xy\}$. 
         There is a unique admissible trading strategy $\widehat{H}$ such that $\widehat{g}=(\widehat{H}\cdot S)_T$.
   \item The following equality is verified for the solutions of the primal and dual problems:
         \begin{equation}\label{optequ}
          \big\langle \widehat{Q}^r_{\hat{y}}, x+\big(\widehat{H}\cdot S\big)_T+e_T\big\rangle 
             = \big\langle \widehat{Q}_{\hat{y}}, x+\big(\widehat{H}\cdot S\big)_T+e_T\big\rangle
             = x+\big\langle \widehat{Q}_{\hat{y}}, e_T\big\rangle.
         \end{equation}
  \end{enumerate}
\end{theorem}
\begin{remark}\label{re25}
 \begin{enumerate}[(i)]
  \item Since the random variable $x+\big(\widehat{H}\cdot S\big)_T+e_T$ in Theorem \ref{thmcsw} is uniformly bounded from below, then $\big\langle\widehat{Q}_{\hat{y}}, x+\big(\widehat{H}\cdot S\big)_T+e_T\big\rangle$ is well defined by
$$
\big\langle\widehat{Q}_{\hat{y}}, x+\big(\widehat{H}\cdot S\big)_T+e_T\big\rangle:=\lim_{M\rightarrow \infty}\big\langle\widehat{Q}_{\hat{y}}, \big(x+\big(\widehat{H}\cdot S\big)_T+e_T\big)\wedge M\big\rangle,
$$  
although it is not necessarily an element in $\massL^\infty$.
  \item From the construction of the primal solution, one can see that $\frac{d\widehat{Q}^r_{\hat{y}}}{d\massP}>0$, $\massP$-a.s., so that $\widehat{Q}^r_{\hat{y}}\sim \massP$. 
  \item The equality of optimality (\ref{optequ}) shows that the purely finitely additive part $\widehat{Q}^s_{\hat{y}}$ ``concentrates'' its mass on the sets, 
$$
\left\{x+\big(\widehat{H}\cdot S\big)_T+e_T<\tfrac{1}{n}\right\},\ {\rm for\ any}\ n\in \mathbb{N}.
$$
 \end{enumerate}
\end{remark}

\section{Revisit the dual problem in \cite{CSW01}}
\noindent In this section, we will present our main result, i.e., in the Brownian framework, the countably additive part $\widehat{Q}^r$ of any dual optimizer $\widehat{Q}\in (\massL^\infty)^*$ obtained in \cite{CSW01} can be attained by the terminal value of a local martingale $\widehat{Y}$ which belongs to the set of all supermartingale deflators, defined by
\begin{align*}
\mathcal{Y}(1):=\{Y\ &{\rm is\ a\ positive\ semimartingale:}\ Y_0=1,\\
 &XY\ {\rm is\ a\ supermartingale,\ for\ any}\ X\in \mathcal{X}(1)\},
\end{align*}
where 
$$
\mathcal{X}(1):=\{1+(H\cdot S): 1+(H\cdot S)_t\geq 0,\ 0\leq t\leq T\}.
$$

%The dual problem (\ref{dp}) becomes the one considered in \cite{KS99} if $e_T\equiv0$. However, the dual domain in \cite{KS99} is a little different from  (\ref{dualdom}), which is defined as the set of terminal values of all supermartingle deflators (see (\ref{dualdomain})) and notably the authors prove that for any $y>0$, the value $v(y)$ of the dual problem can be approached by choosing a sequence of ELMMs (see Proposition 3.2 in \cite{KS99}). In the case with random endowment, we have a similar observation, however, we need the following assumption instead of Assumption \ref{uass}:

We first observe that the dual optimizer for the problem (\ref{dp}) can be approximated by a sequence of ELMMs.
%\begin{assumption}\label{uasss}
%$U(\infty)<\infty$. (Question: can be relaxed?)
%\end{assumption}
\begin{proposition}\label{prodm}
Let Assumptions \ref{NA}, \ref{U(x)assumption}, \ref{uass} hold. Let $\hat{y}:=u'(x)$. If $\widehat{Q}_{\hat{y}}$ is a dual optimizer (denoted by $\widehat{Q}$ for short) for the problem (\ref{dp}), then there exists a sequence $\{\massQ^n\infseq$ of ELMMs, such that
\begin{equation}\label{qet}
\frac{d\massQ^n}{d\massP}\rightarrow \frac{d{\widehat{Q}^r}}{d\massP},\ \massP-a.s.\ {\rm and}\ 
\langle \massQ^n, e_T\rangle\longrightarrow \langle \widehat{Q}, e_T\rangle,\ {\rm as}\ n\rightarrow \infty.
\end{equation}
% we have$\{\massQ^n\infseq$
%$$
%v(y):=\inf_{Q\in \cD}\left\{\massE\left[V\left(y\frac{dQ^r}{d\massP}\right)\right]+y\langle Q, e_T\rangle\right\}=\inf_{\massQ\in \cM}\left\{\massE\left[V\left(y\frac{d\massQ}{d\massP}\right)\right]+y\langle \massQ, e_T\rangle\right\}.
%$$
\end{proposition}
\begin{proof}
First, we claim that $\cD$ is the weak-star closure of $\mathcal{M}$, which can be regarded as a subset of $\linftys$ via the canonical embedding. For the convenience of the reader, we shall briefly prove this claim.  Indeed, $\cD \supseteq  \overline{\mathcal{M}}^{\sigma((L^\infty)^*, L^\infty)}$ is trivial. To show the equality, we suppose, contrary to the claim, there exists a point $\overline{Q}\in \cD$ but
$
\overline{Q}\notin\overline{\mathcal{M}}^{\sigma((L^\infty)^*, L^\infty)}
$.
From \cite{CSW01}, $\cD$ is compact, thus by the Hahn-Banach separation theorem, one can find a function $f\in L^\infty$ and a constant $\alpha$ such that 
$$
\langle Q, f\rangle \leq \alpha,\ {\rm for\ all}\ Q\in \overline{\mathcal{M}}^{\sigma((L^\infty)^*, L^\infty)},$$
but
\begin{equation}\label{contrad}
\langle \overline{Q}, f\rangle >\alpha.
\end{equation}
Applying the superreplication theory, we conclude $f-\alpha\in \mathcal{C}$ and thus, from the definition of $\mathcal{D}$,  we have $\langle \overline{Q}, f\rangle \leq \alpha$, which is in contradiction to (\ref{contrad}). \\

Then, it follows from Corollary \ref{l1linfty} that we can find a sequence $\{\massQ^n\infseq\subset\mathcal{M}$, such that (\ref{qet}) holds.
%\begin{equation}\label{qet}
%\frac{d\massQ^n}{d\massP}\rightarrow \frac{d{\widehat{Q}^r}}{d\massP},\ \massP-a.s.\ {\rm and}\ 
%\langle \massQ^n, e_T\rangle\longrightarrow \langle \widehat{Q}, e_T\rangle,\ {\rm as}\ n\rightarrow \infty.
%\end{equation}
%Moreover, the utility function $U$ satisfies reasonable asymptotic elasticity, then one can follow the proof of Lemma 3.2 in \cite{KS99} to see that 
%$
%\{(V(\frac{d\massQ^n}{d\massP}))^-\}_{n=1}^\infty
%$
%is uniformly integrable. Taking with the fact that $V(0)=U(\infty)<\infty$, we obtain
%$$
%\massE\left[V\left(y\frac{d\massQ^n}{d\massP}\right)\right]\longrightarrow \massE\left[V\left(y\frac{d\widehat{Q}^r}{d\massP}\right)\right],\ {\rm as}\ n\rightarrow \infty,
%$$
%which ends the proof.
\end{proof}
%\begin{remark}
%The above proposition shows that the value $v(y)$ in (\ref{dp}) can be approached by a sequence from $\mathcal{M}$, indeed, this sequence can be chosen as a minimizing one. 
%\end{remark}
\begin{remark}
For any cluster point $Q^\star$ of the sequence $\{\massQ^n\infseq$ in Proposition \ref{prodm}, $Q^\star$ is a dual optimizer for (\ref{dp}) by Proposition A.1 in \cite{CSW01}.
\end{remark}

Letting $\{\massQ^n\}_{n=1}^\infty$ be the sequence chosen in Proposition \ref{prodm}, define for each $n\in\NN$
$$
Y^n_t:=\massE^{\massP}\left[\frac{d\massQ^n}{d\massP}\bigg|\mathcal{F}_t\right],
$$
which is the density process of $\massQ^n$ and is a strictly positive martingale.\\

We recall the definition of optional strong supermartingales. These processes are introduced by Mertens \cite{Mer72} as a generalization of c\`adl\`ag supermartingales. We also refer to Appendix I of \cite{DM82} for more properties of these processes. 
\begin{definition}
A real-valued stochastic process $Y= (Y_t)_{0 \leq t \leq T}$ is called an
{\it optional strong supermartingale}, if
\begin{enumerate}
 \item $Y$ is optional;
 \item $Y_{\tau}$ is integrable for every $[0,T]$-valued stopping time
$\tau$;
 \item For all stopping times $\sigma$ and $\tau$ with $0 \leq \sigma \leq
\tau \leq T$, we have
    $$Y_{\sigma} \geq \massE \left[ Y_{\tau} | \mathcal F_{\sigma} \right].$$
\end{enumerate}
\end{definition} 

By Theorem 2.7 in \cite{CS15strong}, there exists a sequence $\{\widetilde{Y}^n\}_{n=1}^\infty$ of convex combinations $\widetilde{Y}^n\in \conv(Y^n, Y^{n+1}, \ldots)$, and a non-negative optional strong supermartingale $\widehat{Y}$ (not necessarily c\`adl\`ag),
%, see Definition 2.5 and related remarks in \cite{CS14}), 
such that for every $[0, T]$-valued stopping time $\sigma$, we have
\begin{equation}\label{pro}
\widetilde{Y}^n_\sigma\stackrel{\massP}{\longrightarrow}  \widehat{Y}_\sigma,\ {\rm as}\ n\rightarrow \infty.
\end{equation}
Obviously,  
$$\frac{d\widetilde{\massQ}^n}{d\massP}=\widetilde{Y}^n_T\longrightarrow \widehat{Y}_T=\frac{d{\widehat{Q}}^r}{d\massP},\ \massP-a.s.,\ {\rm as}\ n\rightarrow \infty, $$
where $d\widetilde{\massQ}^n=\widetilde{Y}^n_Td\massP$. In the remainder of this paper, our main goal is to show the claim that
\begin{equation}\label{goal}\widehat{Y}\ {\rm is\ a\ local\ martingale,}
\end{equation} 
under the following assumption:
\begin{assumption}\label{assfil}
 The underlying filtration $(\mathcal{F}_t)_{0\leq t\leq T}$ is generated by a Brownian motion.
%satisfies that every $\cF$-stopping time is predictable.
%for  the augmented filtration of $(\mathcal{F}^0)_{0\leq t\leq T}$, where $(\mathcal{F}^0)_{0\leq t\leq T}$ is the raw filtration generated by some Feller process. 
\end{assumption}
Once the claim (\ref{goal}) is verified, we know from the above assumption that $\widehat{Y}$ is continuous and thus a (c\`adl\`ag) supermartingale. By a similar argument as in the proof of Lemma 4.1 in \cite{KS99}, namely, for any $X\in \mathcal{X}(1)$, applying Theorem 2.7 in \cite{CS15strong} again, one can see that $X\widehat{Y}$ is still a supermartingale, which implies $\widehat{Y}\in\mathcal{Y}(1)$.

\begin{remark}
One can also apply the well-known Fatou-convergence result (see Lemma 5.2 in \cite{FK97}) to construct $\widehat{Y}'$ as the Fatou limit of $\{\widetilde{Y}^n\}_{n\in \mathbb{N}}$, whose terminal value is exactly the density $\frac{d{\widehat{Q}}^r}{d\massP}$. Although the process $\widehat{Y}'$ constructed in this way is certainly c\`adl\`ag, 
%we would   
%of the  From the optimality, we see $\widehat{Y}_T$ equals the dual optimizer.
 %However, we would 
 %lack of information on the convergence for the intermediate time between 0 and $T$, namely, 
 yet (\ref{pro}) may fail. Therefore, the advantage of the result in \cite{CS15strong} is that we could find a unified sequence which is not only the limit of $\widetilde{Y}^n$ at the terminal time $T$ but also at any intermediate time.
 %  in the sense of (\ref{pro}), and 
Particularly, one can pick a subsequence such that the convergence holds $\massP$-a.s.~at countably many times.
 % but the disadvantage is that the limit process $\widehat{Y}$ may not be c\`adl\`ag. 
 %As eventually we can prove that the strong optional supermatingale $\widehat{Y}$ is a local martingale, the right continuity can be automatically derived from the nature of the filtration. 
 Note that the difference between the two kinds of limit is only on the graph of countably many stopping times (see Section A.1 in \cite{CSY15}).
\end{remark}

%We shall prove the assertion (\ref{goal}) under the following assumptions:
%\begin{assumption}\label{assst}
%The stock-price process $S$ is continuous.
%\end{assumption}

\begin{remark}
 %We remark that our main result also holds with a filtration generated by a continuous Hunt process, 
 % as our proof only relies on the following property of such a filtration: every $\cF$-stopping time is predictable. 
Under the above assumption, every local martingale has a continuous modification. % (cf. P30 in \cite{CW90}). 
In particular, from Assumption \ref{NA}, the stock-price process $S$ in our setting is indeed continuous. It is not clear to us whether this assumption  is really necessary for the following theorem or it could be weakened. We leave this as an open question. 
%However, it is not clear to us that once this property fails whether the result below still holds true.
% verifies the assumption above. 
%(ii) The filtration $\mathcal{F}$ satisfies the usual conditions; 
%(iii) $\mathcal{F}^0$ is countably generated, namely, $\mathcal{F}^0=\sigma(\{A_1, A_2,\ldots\})$, where $A_1$, $A_2$, $\ldots$ are subsets of $\Omega$. The space $\massL^1(\Omega, \cF, \massP)$ and $\massL^1(\Omega, \cF^0, \massP)$ are equivalent and this space is separable;
%(iii) Every $\cF$-stopping time is predictable. 
\end{remark}
%\begin{remark}
%Noticing that the py $\cF$-stopping time is predictable is  
%\end{remark}
Now, we are ready to state our main result. Its proof is postponed to the next section.
\begin{theorem}\label{main}
Under Assumptions \ref{NA}, \ref{U(x)assumption}, \ref{uass}, \ref{assfil}, the process $\widehat{Y}$ defined in (\ref{pro}) is a local martingale and thus, the regular part $\widehat{Q}^r$ of any dual optimizer obtained in \cite{CSW01} can be attained by a local martingale, which belongs to $\mathcal{Y}(1)$.
\end{theorem}
%\begin{remark}
%With a discontinuous stock-price process $S$ in the general semi-martingale framework, the counterexample for continuous-time model even without the random endowment can be constructed based on Example 5.1' in \cite{KS99}: define $S_t\equiv S_0$, $0\leq t<T$, and $S_T=S_1$, which means that the stock-price process jumps suddenly at time $T$, where $(S_0, S_1)$ is the one period process defined in \cite{KS99}. In this example, the dual optimizer is associated with a process $\widehat{Y}$ that is constant $1$ on $[0, T)$, however, jumps to $S^{-1}_T$ at time $T$, where $\massE^\massP[S^{-1}_T]<1$.
%\end{remark}

\section{Proof of Theorem \ref{main}}
\noindent In this section, we shall prove Theorem \ref{main}. We break the proof into three main steps. In the sequel, each subsection stands for a step. 

\subsection{The fictional optimal wealth process} In a first stage, we construct a fictional optimal wealth process $\widehat{W}$, which attains the optimal terminal value $x+\big(\widehat{H}\cdot S\big)_T+e_T$. Then, we look for a sequence of stopping times, such that at each stopping time, the process $\widehat{W}$ is bounded away from 0.

\vspace{2mm}

Define
$$
\widetilde{W}^n_t:=x+(\widehat{H}\cdot S)_t+\massE^{\widetilde{\massQ}^n}[e_T|\mathcal{F}_t],\quad 0\leq t\leq T,\quad {\rm for\ all}\ n\in \mathbb{N}.
$$
Since $\mathcal{M}$ is closed with respect to convex combination, $\widetilde{\massQ}^n$ is still an ELMM, so that $\widetilde{W}^n$ is a $\widetilde{\massQ}^n$-supermartingale. It follows from the optimality of $\widehat{H
}$ that $\widetilde{W}^n_T=x+(\widehat{H}\cdot S)_T+e_T>0$, $\massP$-a.s., which holds also $\widetilde{\massQ}^n$-a.s., since $\massP$ and $\widetilde{\massQ}^n$ are equivalent. By Theorem VI-17 in \cite{DM82}, one can deduce that 
\begin{equation}\label{wmin}
 \inf_{0\leq t\leq T}\widetilde{W}^n_t>0,\ \widetilde{\massQ}^n-a.s., 
\end{equation}
which holds also $\massP$-a.s.

\vspace{2mm}

Consider the process $$\widetilde{Y}^n_t\widetilde{W}^n_t=\widetilde{Y}^n_t\big(x+(\widehat{H}\cdot S)_t\big)+\massE^{\massP}\big[\widetilde{Y}^n_Te_T|\mathcal{F}_t\big],\quad 0\leq t\leq T, $$
which is obviously a strictly positive $\massP$-supermartingale. Applying Theorem 2.7 in \cite{CS15strong} again, there exists a sequence of convex combinations of $\{\widetilde{Y}^n\infseq$, denoted still by $\{\widetilde{Y}^n\infseq$, and a non-negative optional strong supermartingale $\widehat{Z}$, such that for every $[0,T]$-valued stopping time $\sigma$ we have
\begin{equation}\label{wcon}
\widetilde{Y}^n_\sigma\widetilde{W}^n_\sigma\stackrel{\massP}{\longrightarrow}  \widehat{Z}_\sigma,\ {\rm as}\ n\rightarrow \infty.
\end{equation}
It is evident that (\ref{pro}) still holds for $\{\widetilde{Y}^n\infseq$ as well.

\begin{proposition} 
  The process $\widehat{W}:={\widehat{Z}}/{\widehat{Y}}$ is well defined. 
\end{proposition}

\begin{proof}
 Since $\widehat{Y}_T=\frac{d{\widehat{Q}}^r}{d\massP}$, as stated in Remark \ref{re25} (ii), we have that $\widehat{Y}_T>0$, $\massP$-a.s. Then, one can employ the same argument as (\ref{wmin}) to deduce  (see Theorem VI-17 and Appendix I Remark 5 in \cite{DM82})
\begin{equation}\label{lby}
\inf_{0\leq t\leq T}\widehat{Y}_t>0,\ \massP-a.s.,
\end{equation}
which implies that $\widehat{W}$ is well defined. 
\end{proof}
\begin{proposition}\label{zmart}
The process $\widehat{Z}$ is a martingale and from Assumption \ref{assfil} it has a continuous modification.
\end{proposition}
\begin{proof}
Note that 
$$
\widehat{W}_T=\widetilde{W}^n_T=x+(\widehat{H}\cdot S)_T+e_T,\ {\rm for\ all}\ n\in \mathbb{N}.
$$
We have, from (\ref{optequ}) and (\ref{qet}),
$$
\widehat{Z}_0=\lim_{n\rightarrow \infty} \widetilde{Y}^n_0\widetilde{W}^n_0 =x+\lim_{n\rightarrow \infty} \langle \widetilde{\massQ}^n, e_T\rangle=x+\langle \widehat{Q}, e_T\rangle =\massE^\massP\big[\widehat{Y}_T\widehat{W}_T\big]=\massE^\massP\big[\widehat{Z}_T\big],
$$
from which we conclude that the process $\widehat{Z}$ is a martingale, since we have already known it is an optional strong supermartingale during the construction.
\end{proof}
\begin{proposition}\label{stop}
There exists a sequence of stopping times $\{\tau_k\}_{k=1}^\infty$,  such that $$\widehat{W}_{t\wedge \tau_k}\geq \tfrac{1}{k},\ {\rm for\ all}\ t\in [0, T],$$ and $\massP(\tau_k=T)\nearrow 1$, as $k\rightarrow \infty$. 
\end{proposition}
\begin{proof}
Although $\widehat{Y}$ is only an optional strong supermartingale, we can always apply the martingale inequality to show that
$\sup_{0\leq t\leq T}\widehat{Y}_t<\infty,\ \massP$-a.s. (see P395, Appendix I-3 in \cite{DM82}). On the other hand, thanks to proposition \ref{zmart}, we could proceed with the same argument as (\ref{wmin}) to obtain $\inf_{0\leq t\leq T}\widehat{Z}_t>0,\ \massP$-a.s. Clearly, we now have 
$$
\inf_{0\leq t\leq T}\widehat{W}_t>0,\ \massP-a.s.
$$
Without loss of generality, we assume $\widehat{Z}_t=\widehat{Z}_{T}$, for $t\geq T$. Define 
 \begin{equation}\label{stopping}
   \sigma_k:=\inf\left\{t>0: \widehat{W}_t<\tfrac{1}{k}\right\},\ {\rm for}\ k\in \mathbb{N}, 
 \end{equation}
which goes to infinity. From Assumption \ref{assfil}, all the stopping times defined above are predictable. Therefore, for each $k$, we can choose a sequence $\sigma_{k, m}\rightarrow \sigma_k$ and $\sigma_{k, m}<\sigma_k$, whenever $\sigma_k>0$. Define $\tau_k:=\sigma_{k, m_k}\wedge T$, where
 $$\massP\left(|\sigma_{k, m_k}-\sigma_{k}|>\tfrac{1}{2^k}\right)<\tfrac{1}{2^k}. $$
The sequence $\{\tau_k\}_{k=1}^\infty$ yields the desired result.
\end{proof}

\vspace{2mm}

\subsection{The fictional process for the random endowment} In the sequel, fix $k\in \mathbb{N}$ and denote by $\tau=\tau_k$. We shall first decompose $\widehat{W}$ and obtain a fictional process for the random endowment $e_T$. Then, we construct a dual optimizer $Q^\star$ and prove that the random variable $e_\tau$ is the conditional expectation of $e_T$ under $Q^\star$. 

\vspace{2mm}

It follows from (\ref{pro}) and (\ref{wcon}) that for every $[0,T]$-valued stopping time $\tau$,
\begin{equation}\label{wcon2}
\widetilde{W}^n_\tau\stackrel{\massP}{\longrightarrow}  \widehat{W}_\tau,\ {\rm as}\ n\rightarrow \infty.
\end{equation}
Then, we rewrite the process $\widehat{W}$ as
$$
\widehat{W}_t=x+(\widehat{H}\cdot S)_t+e_t,\ 0\leq t\leq T,
$$
where
\begin{equation}\label{proe}
e_t:=\massP-\lim_{n\rightarrow \infty}\ \massE^{\widetilde{\massQ}^n}[e_T|\mathcal{F}_t]
\end{equation}
with
$$
\massE^{\widetilde{\massQ}^n}[e_T|\mathcal{F}_t]= \frac{\massE^{\massP}\big[\widetilde{Y}^n_T e_T|\mathcal{F}_t\big]}{\massE^{\massP}\big[\widetilde{Y}^n_T|\mathcal{F}_t\big]}.
$$
\begin{remark}
In \cite{CSW01}, $e_T$ is indeed associated with a cumulative process $e:=(e_t)_{0\leq t\leq T}$ with $e_0=0$, however, only the terminal value $e_T$ influences the choice of the agent. In our paper, $e:=(e_t)_{0\leq t\leq T}$ is a fictional value process with the terminal value $e_T$, which is constructed by (\ref{proe}) and should be differed from the one in \cite{CSW01}.
\end{remark}

With the stopping time $\tau$, we see that $\big\{\massE^{\massP}\big[\widetilde{Y}^n_T e_T|\mathcal{F}_\tau\big]\big\}_{n=1}^\infty$ and $\big\{\massE^{\massP}\big[\widetilde{Y}^n_T|\mathcal{F}_\tau\big]\big\}_{n=1}^\infty$ are $\massL^1$-bounded, then recalling Koml\'os' lemma, we can find a sequence $\{\overline{Y}^n\}_{n=1}^\infty$ of convex combinations $\overline{Y}^n\in\conv(\widetilde{Y}^n, \widetilde{Y}^{n+1}, \ldots)$ associated with $\overline{\massQ}^n\in\conv(\widetilde{\massQ}^n, \widetilde{\massQ}^{n+1}, \ldots)$, such that $\massP$-a.s., for some $g\in \massL^1(\Omega, \mathcal{F}_\tau, \massP)$,
\begin{align}\label{con2}
   \lim_{n\rightarrow \infty} \massE^{\massP}\big[\overline{Y}^n_T e_T|\mathcal{F}_\tau\big]&=g, \quad
    \lim_{n\rightarrow \infty}\massE^{\massP}\big[\overline{Y}^n_T|\mathcal{F}_\tau\big]=\lim_{n\rightarrow \infty}\overline{Y}^n_\tau=\widehat{Y}_\tau,  \\
     \mbox{ and }  \quad  e_\tau&=\lim_{n\rightarrow \infty}\massE^{\overline{\massQ}^n}[e_T|\mathcal{F}_\tau]=\frac{g}{\widehat{Y}_\tau}. \nonumber
%=\frac{\lim_n \massE^{\massP}[\overline{Y}^n_T e_T|\mathcal{F}_\tau]}{\lim_n\massE^{\massP}[\overline{Y}^n_T|\mathcal{F}_\tau]}=\lim_n\massE^{\overline{\massQ}^n}[e_T|\mathcal{F}_\tau].
\end{align}

\begin{remark} The random variables $\{\massE^{\overline{\massQ}^n}[e_T|\mathcal{F}_\tau]\infseq$ and $e_\tau$ are elements in $\massL^\infty(\Omega, \mathcal{F}_\tau, \massP)$, and 
\begin{equation*}
\massE^{\overline{\massQ}^n}[e_T|\mathcal{F}_\tau]\stackrel{w^*}{\longrightarrow}e_\tau,\ {\rm as}\ n\rightarrow \infty.
\end{equation*}
Indeed, for each step function $\xi^m\in \mathcal{F}_\tau$, $m\in \mathbb{N}$, it follows from the bounded convergence theorem that $\langle \massE^{\overline{\massQ}^n}[e_T|\mathcal{F}_\tau], \xi^m\rangle\rightarrow \langle e_\tau, \xi^m\rangle$. If $\xi^m\rightarrow \xi$ in $\massL^1(\Omega, \mathcal{F}_\tau, \massP)$, then from the uniform integrability of 
 $$ \big\{\massE^{\overline{\massQ}^n}[e_T|\mathcal{F}_\tau]\xi^m\big\}_{m, n=1}^\infty, $$
we have $\big\langle \massE^{\overline{\massQ}^n}[e_T|\mathcal{F}_\tau], \xi\big\rangle\rightarrow \langle e_\tau, \xi\rangle$.
\end{remark}
%Choosing a subsequence of $\{\overline{Y}^n\infseq$, denoted still by $\{\overline{Y}^n\infseq$, such that
%\begin{equation}\label{conytau}\overline{Y}^n_\tau\longrightarrow \widehat{Y}_\tau,\ \massP-a.s.,\end{equation} 
By Egorov's theorem, there exists an increasing sequence of sets
$\{\Omega_m\}^{\infty}_{m=1}$, such that for each $m$, $\massP{(\Omega_m)}>1-\frac{1}{2^m}$, and $\big\{\overline{Y}^n_\tau\big\}_{n=1}^\infty$ uniformly converges to $\widehat{Y}_\tau$ on $\Omega_m$.  Observing that for each $n\in \mathbb{N}$, $\overline{Y}^n_\tau$ is in $\massL^1_+(\Omega, \mathcal{F}_\tau, \massP)$, we know from Fatou's lemma that $\massE^\massP\big[|\widehat{Y}_\tau|\big]\leq 1$, and thus
\begin{equation}\label{ui}
  \big\{\overline{Y}^n_\tau\big\}_{n=1}^\infty\ {\rm is\ uniformly\ integrable\ on}\ \Omega_m.\end{equation} 
\begin{proposition}
The sequence $\big\{\overline{\massQ}^n\big\}_{n=1}^\infty\subset \mathcal{M}$ associated with $\big\{\overline{Y}^n_\tau\big\}_{n=1}^\infty$ admits a cluster point $Q^\star\in \cD$, such that 
\begin{enumerate}[$(i)$]
 \item $\overline{Y}^n_T\longrightarrow \frac{d{Q^\star}^r}{d\massP},\ \massP$-a.s.; 
 \item $\langle \widehat{Q}, e_T\rangle =\langle Q^\star, e_T\rangle$; 
 \item $Q^\star$ is a dual optimizer for the problem (\ref{dp}).
\end{enumerate} 
\end{proposition}

\begin{proof}
Note $\{\overline{\massQ}^n\infseq\subset \mathcal{M}\subset\cD$. Since $\cD$ is a weak-star compact subset of $\linftys$, the sequence $\{\overline{\massQ}^n\infseq$ admits a cluster point $Q^\star\in\cD$. (i) follows from Proposition A.1 in \cite{CSW01}; (ii) holds because of (\ref{qet}); (i) and (ii) imply (iii). 
\end{proof}

Immediately, we have the following corollary:

\begin{corollary}\label{sing}
The finitely additive measure $Q^\star|_{\mathcal{F}_\tau}$ is countably additive on $\Omega_m$, for each $m\in \mathbb{N}$, where $ Q^\star|_{\mathcal{F}_\tau}$ denotes the restriction of $ Q^\star$ on $\mathcal{F}_\tau$. In other words, $(Q^\star|_{\mathcal{F}_\tau})^s$ vanishes on $\Omega_m$, i.e., for each $A\subset \Omega_m$, $A\in \mathcal{F}_\tau$, $\langle(Q^\star|_{\mathcal{F}_\tau})^s, {\bf 1}_{A}\rangle=0$.
\end{corollary}

Proceeding as in the proof of Corollary 5.2 in \cite{PR14}, we observe that $Q^\star|_{\mathcal{F}_\tau}$ is also a cluster point of the sequence $\{\overline{\massQ}^n|_{\mathcal{F}_\tau}\infseq$, and it follows again from Proposition A.1.~in \cite{CSW01} that

\begin{proposition}\label{lower}
$$\widehat{Y}_\tau=\lim_{n\rightarrow \infty}\massE^{\massP}\big[\overline{Y}^n_T|\mathcal{F}_\tau\big]=\frac{d(Q^\star|_{\mathcal{F}_\tau})^r}{d\massP}, \massP-a.s.$$
\end{proposition}

\begin{remark}
We remark that the choice of the finitely additive measure $Q^\star$ depends on the stopping time $\tau$, which may not be a F\"ollmer finitely additive measure for $\widehat{Y}$ (see Definition 2.6 in \cite{PR14}).
%since it is not uniform in $\tau$  
%Noticing that there is a slight difference between the sense of convergence in (\ref{pro}) and the one in Corollary 5.2 in \cite{PR14}, otherwise the existence of a F\"ollmer finitely additive measure for $\widehat{Y}$ will be immediate. 
\end{remark}

The following corollary is straightforward from (\ref{lby}): 

\begin{corollary}\label{equiv}
$$(Q^\star|_{\mathcal{F}_\tau})^r\sim \massP.$$
\end{corollary}
An argument similar to the one used in Proposition \ref{lower} shows that
\begin{proposition}
$$g=\lim_{n\rightarrow \infty}\massE^{\massP}\big[\overline{Y}^n_Te_T|\mathcal{F}_\tau\big]=\frac{d\big((Q^\star e_T)|_{\mathcal{F}_\tau}\big)^r}{d\massP}, \massP-a.s.,$$
where $Q^\star e_T$ denotes the linear operator $\langle Q^\star, e_T\cdot\rangle\in \linftys$ and 
\begin{align*}
 &\big((Q^\star e_T)|_{\mathcal{F}_\tau}\big)^r:=\big((Q^\star e^+_T)|_{\mathcal{F}_\tau}\big)^r-\big((Q^\star e^-_T)|_{\mathcal{F}_\tau}\big)^r;\\
 &\big((Q^\star e_T)|_{\mathcal{F}_\tau}\big)^s:=\big((Q^\star e^+_T)|_{\mathcal{F}_\tau}\big)^s-\big((Q^\star e^-_T)|_{\mathcal{F}_\tau}\big)^s.
\end{align*}
\end{proposition}

The lemma below is the core of the proof to Theorem \ref{main}:
\begin{lemma}\label{condi}
The random variable $e_\tau$ is the conditional expectation of $e_T$ under $Q^\star$ with respect to $\mathcal{F}_\tau$. In particular, 
$$
\langle Q^\star|_{\mathcal{F}_\tau}, e_\tau\rangle =\langle  Q^\star, e_T\rangle.
$$
\end{lemma}
\begin{proof}
It follows from the boundedness of $e_T$, there exists a unique random variable $\eta\in \mathcal{F}_\tau$ (see Definition 7.1 and Theorem 7.2 in \cite{Lux91}), such that for each $A\in \mathcal{F}_\tau$, 
$$
\langle Q^\star|_{\mathcal{F}_\tau}\eta, {\bf 1}_{A}\rangle=\langle Q^\star e_T, {\bf 1}_{A}\rangle.
$$ 
Our aim now is to prove 
$$
\eta=e_\tau=\frac{\frac{d\big((Q^\star e_T)|_{\mathcal{F}_\tau}\big)^r}{d\massP}}{\frac{d(Q^\star|_{\mathcal{F}_\tau})^r}{d\massP}},\ \massP-a.s.
$$
It is evident that $Q^\star e_T$ is a cluster point of the sequence $\{\overline{\massQ}^ne_T\infseq$, and similar to Corollary \ref{sing}, we know that for each $m\in\NN$, $((Q^\star e_T)|_{\mathcal{F}_\tau})^s$ vanishes on $\Omega_m$. Then, for any $A\subset \Omega_m$, $A\in \mathcal{F}_\tau$, 
\begin{align*}
\left\langle (Q^\star|_{\mathcal{F}_\tau})^r, \frac{\frac{d\big((Q^\star e_T)|_{\mathcal{F}_\tau}\big)^r}{d\massP}}{\frac{d(Q^\star|_{\mathcal{F}_\tau})^r}{d\massP}}{\bf 1}_A \right\rangle
&=\massE^\massP\left[\frac{d\big((Q^\star e_T)|_{\mathcal{F}_\tau}\big)^r}{d\massP}{\bf 1}_A \right]=\big\langle \big((Q^\star e_T)|_{\mathcal{F}_\tau}\big)^r , {\bf 1}_A\big\rangle\\
=\langle (Q^\star e_T)|_{\mathcal{F}_\tau} , {\bf 1}_A\rangle&=\langle Q^\star e_T , {\bf 1}_A\rangle=\langle Q^\star|_{\mathcal{F}_\tau}\eta, {\bf 1}_{A}\rangle=\left\langle \left(Q^\star|_{\mathcal{F}_\tau}\right)^r, \eta{\bf 1}_{A}\right\rangle,
\end{align*}
which implies $\eta=e_\tau$, $(Q^\star|_{\mathcal{F}_\tau})^r$-a.s.~on $\Omega_m$. Thanks to Corollary \ref{equiv}, $\eta=e_\tau$, $\massP$-a.s.~on $\Omega_m$. Letting $m\rightarrow \infty$, we end the proof.
\end{proof}
\subsection{Proof of the main result} In this subsection, we show that  $\{\overline{Y}^n_{\tau}\infseq$ is uniformly integrable so that $\widehat{Y}_{\cdot\wedge \tau}$ is a martingale. Then, substituting $\tau$ by $\tau_k$, $k\in\mathbb{N}$, we can conclude that $\widehat{Y}$ is a local martingale. 

\vspace{2mm}
%We now summarize the proof of Theorem \ref{main}.\\

Let us first consider a dynamic version of (\ref{optequ}): 
\begin{proposition}
\begin{equation*}
 \big\langle Q^\star|_{\mathcal{F}_\tau}, \widehat{W}_\tau\big\rangle=\big\langle Q^\star|_{\mathcal{F}_\tau}, x+(\widehat{H}\cdot S)_\tau+e_\tau\big\rangle =  x+\langle Q^\star, e_T\rangle.
\end{equation*}
\end{proposition}
\begin{proof}
Since $Q^\star\in \mathcal{D}$, by definition, $\langle Q^\star|_{\mathcal{F}_\tau}, (\widehat{H}\cdot S)_\tau\rangle =\langle Q^\star, (\widehat{H}\cdot S)_\tau\rangle\leq 0$. Thus, from the positivity of $(Q^\star|_{\mathcal{F}_\tau})^s$, the martingale property of $\widehat{Y}\widehat{W}$ together with Lemma \ref{condi}, we obtain
$$x+\langle Q^\star, e_T\rangle=\langle(Q^\star|_{\mathcal{F}_\tau})^r, x+(\widehat{H}\cdot S)_\tau+e_\tau\rangle\leq  \langle Q^\star|_{\mathcal{F}_\tau}, x+(\widehat{H}\cdot S)_\tau+e_\tau\rangle \leq x+\langle Q^\star, e_T\rangle.
$$
\end{proof}
%\begin{remark}
%The proof of this proposition can also follow from Remark 4.6 \cite{CSW01}, which states that the primal solution $(\widehat{H}\cdot S)$ has a martingale under the dual optimizer.  
%\end{remark}
Now we can deduce that $\{\overline{Y}^n_\tau\infseq$ is uniformly integrable.

\begin{proposition}\label{yui}
The sequence of random variables $\big\{\overline{Y}^n_\tau\big\}_{n=1}^{\infty}=\big\{\massE^\massP\big[\frac{d\overline{\massQ}^n}{d\massP}|{\cF_\tau}\big]\big\}_{n=1}^{\infty}$ is uniformly integrable and 
$$\overline{Y}^n_{\tau}\stackrel{\massL^1}{\longrightarrow} \widehat{Y}_{\tau},\ {\rm as}\ n\rightarrow \infty.$$
\end{proposition}

\begin{proof}
From the proof the proposition above, we see that 
 $$ \big\langle(Q^\star|_{\mathcal{F}_\tau})^s, x+(\widehat{H}\cdot S)_\tau+e_\tau\big\rangle=0,$$
on the other hand, according to Proposition \ref{stop}, $\widehat{W}_\tau=x+(\widehat{H}\cdot S)_\tau+e_\tau>0$, $\massP$-a.s. Thus, we derive that $(Q^\star|_{\mathcal{F}_\tau})^s\equiv 0$. Recalling that $\|Q^\star\|_{\linftys}=1$, we have
\begin{equation*}\massE^\massP\big[\widehat{Y}_{\tau}\big]=\massE^\massP\left[\frac{d(Q^\star|_{\mathcal{F}_\tau})^r}{d\massP}\right]=1.\end{equation*}
We summarize as follows
$$
\massE^\massP\big[\overline{Y}^n_{\tau}\big]=1\ {\rm and}\ \overline{Y}^n_{\tau}\longrightarrow \widehat{Y}_\tau,\ \massP-a.s.,\ {\rm as}\ n\rightarrow \infty.
$$
The desired result follows by Scheff\'e's lemma (compare with Lemma \ref{keylemma}).
%It follows by Lemma \ref{keylemma} that $\{\overline{Y}^n_\tau\infseq$ is uniformly integrable.
\end{proof}

\noindent \textbf{Proof of Theorem \ref{main}:} For each $\tau_k$ defined in Proposition \ref{stop}, it follows by Proposition \ref{yui} that %$\{\overline{Y}^n_{\tau_k}\infseq$ %is uniformly intergrable and thus, 
$$\overline{Y}^n_{\tau_k}\stackrel{\massL^1}{\longrightarrow} \widehat{Y}_{\tau_k},\ {\rm as}\ n\rightarrow \infty.$$
From the martingale property and (\ref{pro}), we also have for $t\in [0, T]$,
$$\overline{Y}^n_{t\wedge\tau_k}\stackrel{\massL^1}{\longrightarrow} \widehat{Y}_{t\wedge\tau_k},\ {\rm as}\ n\rightarrow \infty,$$
which implies $\widehat{Y}_{\cdot\wedge \tau_k}$ is a martingale. By the definition of $\{\tau_k\}_{k=1}^\infty$, $\widehat{Y}$ is a local martingale. As already discussed in the previous section, $\widehat{Y}$ is consequently a 
%has a continuous modification and for each $X\in\mathcal{X}(1)$, $X\widehat{Y}$ is a supermatingale. Therefore, $\widehat{Y}$ is a 
supermartingale deflator, i.e., $\widehat{Y}$ belongs to $\mathcal{Y}(1)$.\hfill$\square$
%
%and the outline is as follows: first we look for an appropriate increasing sequence of stopping times $\{\tau_k\}_{k=1}^\infty$; then verify that for each $k$, there exists a subsequence of convex combinations of $\{\widetilde{Y}^n_{\tau_k}\infseq$, denoted by $\{\overline{Y}^n_{\tau_k}\infseq$, which is uniformly integrable. So that,
%$$\overline{Y}^n_{\tau_k}\stackrel{\massL^1}{\longrightarrow} \widehat{Y}_{\tau_k},\ {\rm as}\ n\rightarrow \infty,$$
%and thus, by (\ref{pro}), the martingale sequence $\{\overline{Y}^n_{\cdot\wedge\tau_k}\infseq$ converges in $\massL^1$, which yields a true martingale $\widehat{Y}_{\cdot\wedge\tau_k}$. 
\begin{remark} Indeed, for each $k$, the dual optimizer $Q^\star$ we constructed generates an ELMM on $\llbracket 0, \tau_k\rrbracket$ such that the pricing of the fictional random endowment $e_{\tau_k}$ under this measure is exact $\langle \widehat{Q}, e_T\rangle$, in particular, $\langle \widehat{Q}, e_T\rangle$ is an arbitrage-free price for $e_{\tau_k}$.
% is the \massE^{\widehat{Q}}[e_T]$ is a good evaluation for $e_T$... argument with utility indifference pricing or something...
\end{remark}
\begin{remark}\label{inta}
We would like to explain a little about the dynamics of $Q^\star|_{\mathcal{F}_{\cdot}}$. Clearly, the underlying price process $S$ and local martingale $\widehat{Y}$ is continuous, so is the wealth process $\widehat{W}$. Consider a set $A\in \mathcal{F}_\tau$, where $\tau$ is a $[0, T]$-valued stopping time, and suppose $\widehat{W}_\tau$ is strictly positive on $A$. For an infinitesimal $\delta t$ such that $\widehat{W}_{\tau+\delta t}$ cannot suddenly jump to $0$, we can show that if $(Q^\star|_{\mathcal{F}_{\tau}})^s\equiv 0$ on $A$, then $(Q^\star|_{\mathcal{F}_{\tau+\delta t}})^s\equiv 0$ on $A$. Otherwise, 
\begin{align*}
\langle Q^\star|_{\mathcal{F}_{\tau+\delta}}, \widehat{W}_{\tau+\delta}{\bf 1}_A\rangle&= \langle(Q^\star|_{\mathcal{F}_{\tau+\delta}})^r, \widehat{W}_{\tau+\delta}{\bf 1}_A\rangle+\langle(Q^\star|_{\mathcal{F}_{\tau+\delta}})^s, \widehat{W}_{\tau+\delta}{\bf 1}_A\rangle\\
&>\langle(Q^\star|_{\mathcal{F}_{\tau+\delta}})^r, \widehat{W}_{\tau+\delta}{\bf 1}_A\rangle=\langle(Q^\star|_{\mathcal{F}_{\tau}})^r, \widehat{W}_{\tau}{\bf 1}_A\rangle=\langle Q^\star|_{\mathcal{F}_{\tau}}, \widehat{W}_{\tau}{\bf 1}_A\rangle.
\end{align*}
This implies $\big\langle Q^\star, \big((\widehat{H}\cdot S)_{\tau+\delta t}-(\widehat{H}\cdot S)_{\tau}\big){\bf 1}_A\big\rangle>0$, which is a contradiction to the definition of $\cD$.
\end{remark}
\section{Appendix}
\subsection{Some results on $(\massL^\infty)^*$} We state and prove some known results in the space $(\massL^\infty)^*$ for the convenience of the reader. A detailed discussion can be found in \cite{DS67, RR83, YH52}. 

\vspace{2mm}

Let $(\Omega, \mathcal{F}, \massP)$ be the underlying probability space and $(\massL^\infty(\Omega, \mathcal{F}, \massP))^*$ be the dual space of $\massL^\infty(\Omega, \mathcal{F}, \massP)$. Unless otherwise specified, we employ the notation $(\massL^\infty)^*$ and $\massL^\infty$ to denote these two spaces for the simplicity's sake. Denote by $(\massL^\infty)^*_+$ the set of all nonnegative elements in $\linftys$. For any $Q\in \linftysp$, there exists a unique decomposition 
$$
Q=Q^r+Q^s,\ Q^r\geq 0,\ Q^s\geq 0,
$$  
where $Q^r$ is countably additive and called the regular part, $Q^r\ll\massP$, and $Q^s$ is purely finitely additive and called the singular part. In particular, for any $\varepsilon >0$, there exists a set $A_\varepsilon\in \mathcal{F}$, such that $\massP(A_\varepsilon)>1-\varepsilon$ and $\langle Q^s, {\bf 1}_{A_\varepsilon}\rangle=0$.
\begin{proposition}\label{propcont}
Suppose $\widetilde{\cD}$ is a convex subset of $\massL^1_+$, which is also a subset of $\linftysp$ via the canonical embedding. Denote by $\cD$ the weak-star closure of $\widetilde{\cD}$ in $\linftysp$. Then, for each $Q\in \cD$, there exists a sequence $\{Q^n\}_{n=1}^\infty\subset \widetilde{\cD}$, such that $Q^n\rightarrow Q^r$, $\massP$-a.s.,  as $n\rightarrow \infty$. 
\end{proposition}
\begin{proof}
Fixing $Q\in \cD$, for each $n\in \NN$, there exists a set $A_n\in \mathcal{F}$, such that $\massP(A_n)>1-\frac{1}{2^n}$ and $Q^s$ is null on $A_n$. By the definition of $\cD$, we see that $Q$ is a weak-star limit point of $\widetilde{\cD}$ and thus, $Q|_{A_n}\in(\massL^\infty)^*$ is also a weak-star limit point of $\widetilde{\cD}|_{A_n}$, where
$
\widetilde{\cD}|_{A_n}:=\{\widetilde{Q}|_{A_n}: \widetilde{Q}\in \widetilde{\cD}\}
$. From $Q|_{A_n}=Q^r|_{A_n}\in \massL^1$, we know that $Q^r|_{A_n}$ is a weak limit point of $\widetilde{\cD}|_{A_n}$. %i.e. $Q^r|_{A_n}$ belongs to the weak closure of $\widetilde{\cD}|_{A_n}$. 
Moreover, due to the fact that $\widetilde{\cD}|_{A_n}$ is convex, $Q^r|_{A_n}$ belongs to the $\massL^1$-closure of $\widetilde{\cD}|_{A_n}$. Therefore, there exists an element $Q^n\in \widetilde{\cD}$, such that 
$$
\|Q^n|_{A_n}-Q^r|_{A_n}\|_{\massL^1}<\tfrac{1}{2^n}.
$$
Finally,
$$
\|Q^n|_{A_n}-Q^r\|_{\massL^1}\leq 
\|Q^n|_{A_n}-Q^r|_{A_n}\|_{\massL^1}+\|Q^r|_{A_n}-Q^r\|_{\massL^1}\longrightarrow 0,
%<\tfrac{1}{2^n}.
$$
%Observing that for each $m\in \mathbb{N}$, 
%$$
%\|Q^n|_{A_m}-Q^r|_{A_m}\|_{\massL^1}\longrightarrow 0,\ {\rm as}\ n\rightarrow \infty.
%$$
%we can choose a subsequence $\{Q^{n_j}\}_{j=1}^\infty$ by the diagonal argument, such that $Q^{n_j}|_{A_m}\rightarrow Q^r|_{A_m}$, $\massP$-a.s., 
which yields $Q^{n}\rightarrow Q^r$, $\massP$-a.s. up to a subsequence. 
\end{proof}

\begin{corollary} \label{l1linfty}
Suppose $\widetilde{\cD}$ is a convex subset of $\massL^1_+$, which is also a subset of $\linftysp$ via the canonical embedding. Denote by $\cD$ the weak-star closure of $\widetilde{\cD}$ in $\linftysp$. Then, for each $Q\in \cD$, there exists a sequence $\{Q^n\}_{n=1}^\infty\subset \widetilde{\cD}$, such that $Q^n\rightarrow Q^r$, $\massP$-a.s., as $n\rightarrow \infty$, and for countably many $f_i\in \massL^\infty$, $i\in \mathbb{N}$, 
$$
\langle Q^n, f_i\rangle \longrightarrow \langle Q, f_i\rangle,\ {\rm as}\ n\rightarrow \infty.
$$
\end{corollary}
\begin{proof}
For each $n$, define
$$
\widetilde{\cD}^n := \bigcap^n_{i=1}\left\{\widetilde{Q}\in \widetilde{\cD}:  | \langle \widetilde{Q} , f_i\rangle - \langle Q, f_i\rangle | < \tfrac{1}{n}         \right\}.
$$
It is evident that $Q$ belongs to the weak-star closure of $\widetilde{\cD}^n$. Applying the proposition above, one can find a sequence $\{Q^{n, m}\}_{m=1}^\infty\subset \widetilde{\cD}^n$, such that $Q^{n, m}\rightarrow Q^r$, $\massP$-a.s. By the diagonal argument, we complete the proof.
\end{proof}
\begin{remark}
We remark that the assumption that $\widetilde{\cD}\subset \massL^1$ is crucial in the above proposition. For a general subset $\widetilde{D}\subset\linftys$ and an element $Q$ in its weak-star closure $D$, one may not find a sequence $\{Q^n\}_{n=1}^\infty$ from $\widetilde{D}$, such that $(Q^n)^r\rightarrow Q^r$. For instance, $\Omega=[0, 1]$, $\mathcal{F}$ is the Lebesgue sigma-algebra and $\massP$ is the Lebesgue measure. Define $\widetilde{{D}}:=\{Q\in \linftysp: \|Q\|_{\linftys}=1, Q=Q^s\}$, then we find that statement of Proposition \ref{propcont} does not hold, since $\{Q\in \massL^1_+: \|Q\|_{\linftys}=1\}\subset \{Q\in (\massL^\infty)^*_+: \|Q\|_{\linftys}=1\}=D$.
\end{remark}

\subsection{Alternative proof of Proposition 3.2 in \cite{LZ07}}

In Section 3.2 of \cite{LZ07}, the authors show that if the stock-price process $S$ is a continuous semimartingale in the problem without random endowment formulated in \cite{KS99}, then the dual optimizer is associated with a local martingale living in the set of supermartingale deflators. Here, we shall give an alternative proof for this assertion based on the dynamics of the primal and dual solutions. We emphasize that no extra condition on the filtration $\mathcal{F}$ is assumed in this subsection.

\vspace{2mm}

Before presenting the theorem, we first introduce the following lemma, which provides us an abstract structure.
\begin{lemma}\label{keylemma}
Let $\{Y^n\infseq\subset \massL^1_+(\Omega, \cF, \massP)$ and $\{X^n\infseq\subset \massL^0(\Omega, \cF, \massP)$, where for each $n$, $X^n\geq a>0$, $\massP$-a.s. If there exists a pair of random variables $(Y, X)\in \massL^1(\Omega, \cF, \massP)\times \massL^0(\Omega, \cF, \massP)$, such that
$$
Y^n\rightarrow Y \  {\rm and}\ X\leq \liminf_n X^n,\ \massP-a.s.,
$$
$$
\massE^\massP[YX]\geq \liminf_n \massE^\massP[Y^nX^n].
$$
Then, $\{Y^n\infseq$ is uniformly integrable.
\end{lemma}
\begin{proof}
If not, by passing to a subsequence if necessary, there exists $\varepsilon>0$, for each $n\in \mathbb{N}$, there exists $A_n$ such that $\massP(A_n)<\frac{1}{2^n}$ and $$\massE^\massP\left[Y^n{\bf 1}_{A_{n}}\right]\geq \varepsilon.$$
Define $$\eta^n:=Y^n{\bf 1}_{A_n},\ 
\xi^n:=Y^n{\bf 1}_{A_n^c}.
$$
\noindent Then $\xi^n\rightarrow Y$, a.s., while by Fatou's Lemma, we have
\begin{align*}
\massE^\massP\left[YX\right]\leq \liminf_{n\rightarrow \infty}\massE^\massP[\xi^n X^n]= \liminf_{n\rightarrow \infty}\massE^\massP\left[Y^nX^n-\eta^n X^n\right]\leq \massE^\massP\left[YX\right]-a\varepsilon,
\end{align*}
which is a contradiction.
\end{proof}
In \cite{KS99}, the primal value function can be regarded as (\ref{pp}) with $e_T\equiv 0$. On the other hand, the dual domain is defined as the solid subset generated by all terminal values of supermartingale deflators, namely, \begin{equation}\label{dualdomain}\mathcal{D}:=\{{h\in \massL^0_+(\Omega, \mathcal{F}_T, \massP): h \leq Y_T,\ {\rm for\ some}\ Y\in \mathcal{Y}(1)}\}.\end{equation}
Then, the dual problem is formulated by 
\begin{equation}\label{dpks99}
v(y):=\inf_{h\in \cD}\massE^\massP[V(yh)].
\end{equation}
It has been proved that for each $x>0$, and $\hat{y}:=u'(x)$, the value $v(y)$ is attained by a unique dual optimizer $\widehat{h}_{\hat{y}}\in \cD$, denoted by $\widehat{h}$ for short, and the primal solution $\widehat{X}_T=x+(\widehat{H}\cdot S)_T$ can be constructed in terms of $\widehat{h}$. Moreover, % From the optimality of $h$, we know that $h$ should be associated with some supermartingale deflator $\widehat{Y}$. Moreover, the following relation holds
\begin{equation}\label{martpro}
\massE^\massP\left[\widehat{h}\widehat{X}_T\right]=x.
\end{equation}

Instead of Assumption \ref{assfil}, we assume 
\begin{assumption}\label{assst}
The stock-price process $S$ is continuous.
\end{assumption}
\begin{theorem}\label{a1}
Under Assumptions  \ref{NA}, \ref{U(x)assumption}, \ref{uass}, \ref{assst}, the dual optimizer of (\ref{dpks99}) obtained in \cite{KS99} is associated with a supermartingale deflator, which is a local martingale. \footnote{During the revision of this paper, Kramkov and Weston \cite{KW15} also proved a similar result but using a different technique.}
\end{theorem}
\begin{proof}
Since $(\widehat{H}\cdot S)$ is a supermartingale under each $\massQ\in \cM$, similarly to (\ref{wmin}), we obtain $$\inf_{0\leq t\leq T} \widehat{X}_t>0,\ \massP-a.s.$$ By the continuity of $\widehat{X}:=x+(\widehat{H}\cdot S)$, one can thus define a sequence of stopping times as (\ref{stopping}), such that $\widehat{X}_{\sigma_k\wedge T}\geq \frac{1}{k}$. Recalling Proposition 3.2 in \cite{KS99}, for each $y>0$, the value $v(y)$ of the dual problem can be approximated by choosing a minimizing sequence $\{\massQ^n\infseq$ of ELMMs, associated with the density process $\{Y^n\infseq$, such that $Y^n_T\rightarrow \widehat{h}$, $\massP$-a.s. By Theorem 2.7 in \cite{CS15strong}, we could find a sequence of convex combinations of $\{Y^n\infseq$, and an optional strong supermartingale $\widehat{Y}$, such that $Y^n$ converges to $\widehat{Y}$ in the sense of (\ref{pro}). In particular, after passing to a subsequence,  $Y^n_{\sigma_k\wedge T}\rightarrow \widehat{Y}_{\sigma_k\wedge T}$, $Y^n_{T}\rightarrow \widehat{Y}_{T}=\widehat{h}$, $\
\massP$-a.s.~so that $\widehat{Y}_T$ 
is the dual optimizer. Moreover, applying once again Theorem 2.7 in \cite{CS15strong}, we can show that $\widehat{Y}\widehat{X}$ is an optional strong supermartingale, then we can deduce that $\widehat{Y}\widehat{X}$ is a martingale from (\ref{martpro}). Consequently, we arrive at
\begin{align}\label{compare}
\left\{\begin{array}{ll}
\massE^\massP[{Y}^n_{\sigma_k\wedge T}\widehat{X}_{\sigma_k\wedge T}]&\leq x;\\
\massE^\massP[\widehat{Y}_{\sigma_k\wedge T}\widehat{X}_{\sigma_k\wedge T}]&=x.
\end{array}\right.
\end{align}
It follows immediately by Lemma \ref{keylemma} that $\{Y^n_{\sigma_k\wedge T}\infseq$ is uniformly integrable so that $\widehat{Y}_{\cdot\wedge\sigma_k\wedge T}$ is a true martingale. We now can conclude that $\widehat{Y}$ is a local martingale, and thus is c\`adl\`ag. For any $X\in \mathcal{X}(1)$,  applying Theorem 2.7 in \cite{CS15strong} again, one can see that $X\widehat{Y}$ is still a (c\`adl\`ag) supermartingale, which implies $\widehat{Y}\in\mathcal{Y}(1)$.
\end{proof}
\begin{remark} According to \cite{KK07}, the inspection of the proofs in \cite{KS99} reveals that the usual assumption (NFLVR) can be replaced by a weaker one (NUPBR), which is equivalent to that $\cY(1)\neq \emptyset$ or the existence of ELMDs (see e.g.~\cite{Kar12}). In this case, to deduce that $\widehat{Y}$ is a local martingale, we could proceed the same as above with a minimizing sequence of ELMDs and a classical localization argument if necessary.  
\end{remark}
\begin{remark}
Compare with Section 4, the proof for the case of $e_T=0$ is much simpler, and we only need the continuity of the stock-price process rather than the assumption on the underlying filtration.
%since we do not care where the mass is lost by the ``optimal'' supermartingale deflator, and thus we could avoid settling in the much more complicated $\linftys$ space. 
%Moreover, in the case of $e_T=0$, . 
We explain as follows. Firstly, the wealth process $\widehat{X}$ in both lines of (\ref{compare}) is unified, in contrast, for the case with non-trivial $e_T$, our sequence of fictional wealth processes depends on $n$, and it is not easy to find a sequence of stopping times that stop the fictional wealth processes simultaneously to let all of them stay above 0. 
%   formed up by the conditional expectation of $e_T$ under each $\massQ^n$; 
Secondly, by the continuity of $(\widehat{H}\cdot S)$, we can easily stop the process by some $\tau$ and let it be bounded away from 0, while in the other case,  we lack the continuity of $\widehat{W}$.
\end{remark}

\textbf{Acknowledgement:} The authors would like to thank Prof.~Dr.~Walter Schachermayer for his fruitful discussions and suggestions and to thank Oliver Butterley for his careful reading of an earlier version. 
 The authors also appreciate valuable comments from two anonymous referees and Kim Weston for further improvement.
 %, which helped us significantly improve our manuscript.

%\newpage

\bibliography{Dualoptimizer_Rev_sub}
\bibliographystyle{plain}

\end{document}